\documentclass{amsart}
\usepackage{amssymb}
\usepackage{amsfonts}

\newtheorem{theorem}{Theorem}
\theoremstyle{plain}

\newtheorem{definition}{Definition}

\newtheorem{lemma}{Lemma}

\numberwithin{equation}{section}
\input{tcilatex}

\begin{document}
\title[Harmonically convex functions on the co-ordinates]{Hermite-Hadamard
type inequalities for harmonically convex functions on the co-ordinates}
\author{Erhan SET$^{\clubsuit }$}
\address{$^{\clubsuit }$Department of Mathematics, Faculty of Arts and
Sciences, Ordu University, 52200, Ordu, Turkey}
\email{erhanset@yahoo.com}
\author{\.{I}mdat \.{I}\c{s}can$^{\blacktriangledown }$}
\address{$^{\blacktriangledown }$Department of Mathematics, Faculty of Arts
and Sciences, Giresun University, 28100, Giresun, Turkey.}
\email{imdat.iscan@giresun.edu.tr, imdati@yahoo.com}

\begin{abstract}
A function $f:\left[ a,b\right] \subset \mathbb{R}\rightarrow \mathbb{R}$ is
said to be convex if whenever $x,y\in \left[ a,b\right] $ and $t\in \left[
0,1\right] $, the following inequality holds:%
\begin{equation*}
f(tx+(1-t)y)\leq tf(x)+(1-t)f(y)
\end{equation*}%
In recent years, new classes of convex functions have been introduced in
order to generalize the results and to obtain new estimations. We also
introduce the concept of harmonically convex functions on the co-ordinates.
Also, we establish some inequalities of Hermit-Hadamard type as S.S.
Dragomir' s results in Theorem \ref{T12} and other Hermit-Hadamard type
inequalities\ for these classes of functions.
\end{abstract}

\maketitle

\section{Introduction}

Let $f:I\subseteq \mathbb{R\rightarrow R}$ be a convex mapping defined on
the interval $I$ of real numbers and $a,b\in I$, with $a<b.$ The following
double inequality is well known in the literature as the Hermite-Hadamard
inequality \cite{DP}:%
\begin{equation*}
f\left( \frac{a+b}{2}\right) \leq \frac{1}{b-a}\int_{a}^{b}f\left( x\right)
dx\leq \frac{f\left( a\right) +f\left( b\right) }{2}.
\end{equation*}%
In \cite{iscan1}, Iscan gave definition of harmonically convexity as follows:

\begin{definition}
Let $I\subset 
\mathbb{R}
\backslash \left\{ 0\right\} $ be a real interval. A function $%
f:I\rightarrow 
\mathbb{R}
$ is said to be harmonically convex, if \ 
\begin{equation}
f\left( \frac{xy}{tx+(1-t)y}\right) \leq tf(y)+(1-t)f(x)  \label{1-2}
\end{equation}%
for all $x,y\in I$ and $t\in \lbrack 0,1]$. If the inequality in (\ref{1-2})
is reversed, then $f$ is said to be harmonically concave.
\end{definition}

The following result of the Hermite-Hadamard type for harmonically convex
functions holds.

\begin{theorem}
\label{T11}\cite{iscan1} Let $f:I\subset 
\mathbb{R}
\backslash \left\{ 0\right\} \rightarrow 
\mathbb{R}
$ be a harmonically convex function and $a,b\in I$ with $a<b.$ If $f\in
L[a,b]$ then the following inequalities hold 
\begin{equation}
f\left( \frac{2ab}{a+b}\right) \leq \frac{ab}{b-a}\dint\limits_{a}^{b}\frac{%
f(x)}{x^{2}}dx\leq \frac{f(a)+f(b)}{2}.  \label{1-3}
\end{equation}%
The \ above inequalities are sharp.
\end{theorem}

For recent results and new inequalities related to harmonically convex
functions, see \cite{iscan1, iscan2, iscan3}.

Let us now consider a bidemensional interval $\Delta =:\left[ a,b\right]
\times \left[ c,d\right] $ in $\mathbb{R}^{2}$ with $a<b$ and $c<d$. A
function $f:\Delta \rightarrow \mathbb{R}$ is said to be convex on $\Delta $
if the following inequality:%
\begin{equation*}
f(tx+\left( 1-t\right) z,ty+\left( 1-t\right) w)\leq tf\left( x,y\right)
+\left( 1-t\right) f\left( z,w\right)
\end{equation*}%
holds, for all $\left( x,y\right) ,\left( z,w\right) \in \Delta $ and $t\in %
\left[ 0,1\right] .$ A function $f:\Delta \rightarrow \mathbb{R}$ is said to
be on the co-ordinates on $\Delta $ if the partial mappings $f_{y}:\left[ a,b%
\right] \rightarrow \mathbb{R},$ \ $f_{y}\left( u\right) =f\left( u,y\right) 
$ and $f_{x}:\left[ c,d\right] \rightarrow \mathbb{R},$ \ $f_{x}\left(
v\right) =f\left( x,v\right) $ are convex where defined for all $x\in \left[
a,b\right] $ and $y\in \left[ c,d\right] \ $(see \cite[p. 317]{DP}).

Also, in \cite{D}, Dragomir establish the following similar inequality of
Hadamard's type for co-ordinated convex mapping on a rectangle from the
plane $\mathbb{R}^{2}.$

\begin{theorem}
\label{T12} Suppose that $f:\Delta \rightarrow \mathbb{R}$ is co-ordinated
convex on $\Delta .$Then one has the inequalities:%
\begin{equation}
\begin{array}{l}
f\left( \dfrac{a+b}{2},\dfrac{c+d}{2}\right) \\ 
\\ 
\leq \dfrac{1}{2}\left[ \dfrac{1}{b-a}\dint_{a}^{b}f\left( x,\dfrac{c+d}{2}%
\right) dx+\dfrac{1}{d-c}\dint_{c}^{d}f\left( \dfrac{a+b}{2},y\right) dy%
\right] \\ 
\\ 
\leq \dfrac{1}{\left( b-a\right) \left( d-c\right) }\dint_{a}^{b}%
\dint_{c}^{d}f\left( x,y\right) dydx \\ 
\\ 
\leq \dfrac{f\left( a,c\right) +f\left( a,d\right) +f\left( b,c\right)
+f\left( b,d\right) }{4}.%
\end{array}
\label{E8}
\end{equation}%
The above inequalities are sharp.
\end{theorem}

Similar results can be found in \cite{DP}-\cite{Hwang} and \cite{Latif}-\cite%
{sarikaya}.

Motivated by \cite{DP, D, iscan1, sarikaya}, we will explore a new concept
of co-ordinated convex functions which is called harmonically convex
functions on the co-ordinates. By virtue of this new concept, we except to
present another interesting and important Hermite-Hadamard inequalities for
these classes of functions.

\section{Main Results}

\begin{definition}
\label{def1}Let us consider the bidimensional interval $\Delta =\left[ a,b%
\right] \times \left[ c,d\right] $ in $\left( 0,\infty \right) \times \left(
0,\infty \right) $ with $a<b$ and $c<d$. A function $f:\Delta \rightarrow 
\mathbb{R}
$ is said to be harmonically convex on $\Delta $ if the following inequality:%
\begin{eqnarray}
f\left( \frac{xz}{tz+(1-t)x},\frac{yw}{tw+(1-t)y}\right) &=&f\left( \frac{1}{%
\frac{t}{x}+\frac{1-t}{z}},\frac{1}{\frac{t}{y}+\frac{1-t}{w}}\right)
\label{2-1} \\
&\leq &tf(x,y)+(1-t)f(z,w)  \notag
\end{eqnarray}%
holds, for all $(x,y),\ (z,w)\in \Delta $ and $t\in \left[ 0,1\right] $. If
the inequality in (\ref{2-1}) is reversed, then $f$ is said to be
harmonically concave on $\Delta $.
\end{definition}

\begin{definition}
\label{def2}Let us consider the bidimensional interval $\Delta =\left[ a,b%
\right] \times \left[ c,d\right] $ in $\left( 0,\infty \right) \times \left(
0,\infty \right) $ with $a<b$ and $c<d$. A function $f:\Delta \rightarrow 
\mathbb{R}
$ is said to be harmonically convex on the co-ordinates if the partial
mappings $f_{y}:\left[ a,b\right] \rightarrow \mathbb{R}$, $f_{y}(u):=f(u,y)$
and $f_{x}:\left[ c,d\right] \rightarrow 
\mathbb{R}
$, $f_{x}(v):=f(x,v)$ are harmonically convex where defined for all $y\in %
\left[ c,d\right] $ and $x\in \left[ a,b\right] .$
\end{definition}

We note that if $f_{x}$ and $f_{y}$ are convex and nondecreasing function
then $f_{x}$ and $f_{y}$ are harmonically convex. If $f_{x}$ and $f_{y}$ are
harmonically convex and nonincreasing function then $f_{x}$ and $f_{y}$ are
convex. Also, in definition \ref{def1} and definition \ref{def2}, $\left(
0,-\infty \right) \times \left( 0,\infty \right) $ or $\left( 0,-\infty
\right) \times \left( 0,-\infty \right) $ or $\left( 0,\infty \right) \times
\left( 0,-\infty \right) $ instead of $\left( 0,\infty \right) \times \left(
0,\infty \right) $ can be chosen.

The following lemma holds:

\begin{lemma}
Every harmonically convex function $f:\Delta \rightarrow 
\mathbb{R}
$ is harmonically convex on the co-ordinates, but the converse is not
generally true.
\end{lemma}

\begin{proof}
Suppose that $f:\Delta \rightarrow 
\mathbb{R}
$ harmonically convex on $\Delta $. Consider $f_{x}:\left[ c,d\right]
\rightarrow 
\mathbb{R}
$, $f_{x}(v)=f(x,v)$. Then for all $t\in \left[ 0,1\right] $ and $v,w\in %
\left[ c,d\right] $ one has:%
\begin{eqnarray*}
f_{x}\left( \frac{vw}{tw+(1-t)v}\right) &=&f\left( x,\frac{vw}{tw+(1-t)v}%
\right) \\
&=&f\left( \frac{x^{2}}{tx+(1-t)x},\frac{vw}{tw+(1-t)v}\right) \\
&\leq &tf(x,v)+(1-t)f(x,w) \\
&=&tf_{x}(v)+(1-t)f_{x}(w)
\end{eqnarray*}%
which shows the harmonically convexity of $f_{x}$.

The fact that $f_{y}:\left[ a,b\right] \rightarrow \mathbb{R}$, $%
f_{y}(u)=f(u,y)$ is also harmonically convex on $\left[ a,b\right] $ for all 
$y\in \left[ c,d\right] $ goes likewise and we shall omit the details.

Now, consider $f:\left[ 1,3\right] \times \left[ 2,3\right] \rightarrow %
\left[ 0,\infty \right) $ given $f(x,y)=(x-1)(y-2)$. It is obvious that $f$
is harmonically convex on the co-ordinates but is not harmonically convex on 
$\left[ 1,3\right] \times \left[ 2,3\right] $. Indeed, for $(1,3),\ (2,3)\in %
\left[ 1,3\right] \times \left[ 2,3\right] $ and $t\in \left( 0,1\right) $
we have%
\begin{equation*}
f\left( \frac{2}{2t+(1-t)1},\frac{9}{3t+3(1-t)}\right) =f\left( \frac{2}{1-t}%
,3\right) =\frac{1+t}{1-t}
\end{equation*}%
\begin{equation*}
tf(1,3)+(1-t)f(2,3)=0.
\end{equation*}%
Thus for all $t\in \left( 0,1\right) $, we have 
\begin{equation*}
f\left( \frac{2}{2t+(1-t)1},\frac{9}{3t+3(1-t)}\right) >tf(1,3)+(1-t)f(2,3)
\end{equation*}%
which shows that $f$ is not harmonically convex on $\left[ 1,3\right] \times %
\left[ 2,3\right] $.
\end{proof}

The following\ inequalities of Hermite-Hadamard type hold.

\begin{theorem}
\label{2.2} Suppose that $f:\Delta =\left[ a,b\right] \times \left[ c,d%
\right] \subseteq \left( 0,\infty \right) \times \left( 0,\infty \right)
\rightarrow 
\mathbb{R}
$ is harmonically convex on the co-ordinates on $\Delta $. Then one has the
inequalities: 
\begin{eqnarray}
&&f\left( \frac{2ab}{a+b},\frac{2cd}{c+d}\right)  \label{2-2} \\
&\leq &\frac{1}{2}\left[ \frac{ab}{b-a}\dint\limits_{a}^{b}\frac{f\left( x,%
\frac{2cd}{c+d}\right) }{x^{2}}dx+\frac{cd}{d-c}\dint\limits_{c}^{d}\frac{%
f\left( \frac{2ab}{a+b},y\right) }{y^{2}}dy\right]  \notag \\
&\leq &\frac{abcd}{\left( b-a\right) \left( d-c\right) }\dint\limits_{a}^{b}%
\dint\limits_{c}^{d}\frac{f(x,y)}{\left( xy\right) ^{2}}dxdy  \notag \\
&\leq &\frac{1}{4}\left[ \frac{ab}{b-a}\dint\limits_{a}^{b}\frac{f\left(
x,c\right) }{x^{2}}dx+\frac{ab}{b-a}\dint\limits_{a}^{b}\frac{f\left(
x,d\right) }{x^{2}}dx\right.  \notag \\
&&+\left. \frac{cd}{d-c}\dint\limits_{c}^{d}\frac{f\left( a,y\right) }{y^{2}}%
dy+\frac{cd}{d-c}\dint\limits_{c}^{d}\frac{f\left( b,y\right) }{y^{2}}dy%
\right]  \notag \\
&\leq &\frac{f(a,c)+f(a,d)+f(b,c)+f(b,d)}{4}.  \notag
\end{eqnarray}%
The above inequalities are sharp.
\end{theorem}

\begin{proof}
Since $f:\Delta \rightarrow 
\mathbb{R}
$ is harmonically convex on the co-ordinates it follows that the mapping $%
h_{x}:\left[ c,d\right] \rightarrow 
\mathbb{R}
$, $h_{x}(y)=f(x,y)$ is harmonically convex on $\left[ c,d\right] $ for all $%
x\in \left[ a,b\right] .$ Then by inequality (\ref{1-3}) one has: 
\begin{equation*}
h_{x}\left( \frac{2cd}{c+d}\right) \leq \frac{cd}{d-c}\dint\limits_{c}^{d}%
\frac{h_{x}(y)}{y^{2}}dy\leq \frac{h_{x}(c)+h_{x}(d)}{2},\ x\in \left[ a,b%
\right] .
\end{equation*}%
That is,%
\begin{equation*}
f\left( x,\frac{2cd}{c+d}\right) \leq \frac{cd}{d-c}\dint\limits_{c}^{d}%
\frac{f(x,y)}{y^{2}}dy\leq \frac{f(x,c)+f(x,d)}{2},\ x\in \left[ a,b\right] .
\end{equation*}%
Integrating this inequality on $\left[ a,b\right] $, we have:%
\begin{eqnarray}
&&\frac{ab}{b-a}\dint\limits_{a}^{b}\frac{f\left( x,\frac{2cd}{c+d}\right) }{%
x^{2}}dx  \label{2-2a} \\
&\leq &\frac{abcd}{\left( b-a\right) \left( d-c\right) }\dint\limits_{a}^{b}%
\dint\limits_{c}^{d}\frac{f(x,y)}{\left( xy\right) ^{2}}dxdy  \notag \\
&\leq &\frac{1}{2}\left[ \frac{ab}{b-a}\dint\limits_{a}^{b}\frac{f\left(
x,c\right) }{x^{2}}dx+\frac{ab}{b-a}\dint\limits_{a}^{b}\frac{f\left(
x,d\right) }{x^{2}}dx\right] .  \notag
\end{eqnarray}%
By a similar argument applied for the mapping $h_{y}:\left[ a,b\right]
\rightarrow 
\mathbb{R}
$, $h_{y}(x)=f(x,y)$ we get%
\begin{eqnarray}
&&\frac{cd}{d-c}\dint\limits_{c}^{d}\frac{f\left( \frac{2ab}{a+b},y\right) }{%
y^{2}}dy  \label{2-2b} \\
&\leq &\frac{abcd}{\left( b-a\right) \left( d-c\right) }\dint\limits_{a}^{b}%
\dint\limits_{c}^{d}\frac{f(x,y)}{\left( xy\right) ^{2}}dxdy  \notag \\
&\leq &\frac{1}{2}\left[ \frac{cd}{d-c}\dint\limits_{c}^{d}\frac{f\left(
a,y\right) }{y^{2}}dy+\frac{cd}{d-c}\dint\limits_{c}^{d}\frac{f\left(
b,y\right) }{y^{2}}dy\right] .  \notag
\end{eqnarray}%
Summing the inequalities (\ref{2-2a}) and (\ref{2-2b}), we get the second
and the third inequality in (\ref{2-2}).

By the inequality (\ref{1-3}) we also have:%
\begin{equation*}
f\left( \frac{2ab}{a+b},\frac{2cd}{c+d}\right) \leq \frac{ab}{b-a}%
\dint\limits_{a}^{b}\frac{f\left( x,\frac{2cd}{c+d}\right) }{x^{2}}dx
\end{equation*}%
and%
\begin{equation*}
f\left( \frac{2ab}{a+b},\frac{2cd}{c+d}\right) \leq \frac{cd}{d-c}%
\dint\limits_{c}^{d}\frac{f(\frac{2ab}{a+b},y)}{y^{2}}dy
\end{equation*}%
which give, by addition, the first inequality in (\ref{2-2}).

Finally, by the inequality (\ref{1-3}) we also have:%
\begin{eqnarray*}
\frac{ab}{b-a}\dint\limits_{a}^{b}\frac{f\left( x,c\right) }{x^{2}}dx &\leq &%
\frac{f(a,c)+f(b,c)}{2} \\
\frac{ab}{b-a}\dint\limits_{a}^{b}\frac{f\left( x,d\right) }{x^{2}}dx &\leq &%
\frac{f(a,d)+f(b,d)}{2} \\
\frac{cd}{d-c}\dint\limits_{c}^{d}\frac{f\left( a,y\right) }{y^{2}}dy &\leq &%
\frac{f(a,d)+f(a,c)}{2}
\end{eqnarray*}%
and%
\begin{equation*}
\frac{cd}{d-c}\dint\limits_{c}^{d}\frac{f\left( b,y\right) }{y^{2}}dy\leq 
\frac{f(b,c)+f(b,d)}{2}
\end{equation*}%
which give, by addition, the last inequality in (\ref{2-2}).

If in (\ref{2-2}) we choose $f(x)=1$, then (\ref{2-2}) becomes an equality,
which shows that (\ref{2-2}) are sharp.
\end{proof}

\begin{lemma}
\label{2.3}Let $f:\Delta =\left[ a,b\right] \times \left[ c,d\right]
\subseteq \left( 0,\infty \right) \times \left( 0,\infty \right) \rightarrow 
\mathbb{R}
$ be a partial differentiable mapping on $\Delta $ with $a<b$ and $c<d$. If $%
\frac{\partial ^{2}f}{\partial t\partial s}\in L(\Delta )$, then the
following equality holds%
\begin{eqnarray*}
&&\frac{f(a,c)+f(a,d)+f(b,c)+f(b,d)}{4}+\frac{abcd}{\left( b-a\right) \left(
d-c\right) }\dint\limits_{a}^{b}\dint\limits_{c}^{d}\frac{f(x,y)}{\left(
xy\right) ^{2}}dxdy \\
&&-\frac{1}{2}\left[ \frac{ab}{b-a}\dint\limits_{a}^{b}\frac{f\left(
x,c\right) }{x^{2}}dx+\frac{ab}{b-a}\dint\limits_{a}^{b}\frac{f\left(
x,d\right) }{x^{2}}dx\right. \\
&&\left. +\frac{cd}{d-c}\dint\limits_{c}^{d}\frac{f\left( a,y\right) }{y^{2}}%
dy+\frac{cd}{d-c}\dint\limits_{c}^{d}\frac{f\left( b,y\right) }{y^{2}}dy%
\right] \\
&=&\frac{abcd(b-a)(d-c)}{4}\dint\limits_{0}^{1}\dint\limits_{0}^{1}\frac{%
(1-2t)(1-2s)}{\left( A_{t}B_{s}\right) ^{2}}\frac{\partial ^{2}f}{\partial
t\partial s}\left( \frac{ab}{A_{t}},\frac{cd}{B_{s}}\right) dtds,
\end{eqnarray*}%
where $A_{t}=tb+(1-t)a$ and $B_{s}=sd+(1-s)c.$
\end{lemma}

\begin{proof}
By integration by parts, we get%
\begin{equation}
\frac{abcd(b-a)(d-c)}{4}\dint\limits_{0}^{1}\dint\limits_{0}^{1}\frac{%
(1-2t)(1-2s)}{\left( A_{t}B_{s}\right) ^{2}}\frac{\partial ^{2}f}{\partial
t\partial s}\left( \frac{ab}{A_{t}},\frac{cd}{B_{s}}\right) dtds
\label{2-3a}
\end{equation}%
\begin{eqnarray*}
&=&\frac{cd(d-c)}{4}\dint\limits_{0}^{1}\frac{(2s-1)}{B_{s}^{2}} \\
&&\times \left\{ \left. (1-2t)\frac{\partial f}{\partial s}\left( \frac{ab}{%
A_{t}},\frac{cd}{B_{s}}\right) \right\vert _{0}^{1}+2\dint\limits_{0}^{1}%
\frac{\partial f}{\partial s}\left( \frac{ab}{A_{t}},\frac{cd}{B_{s}}\right)
dt\right\} ds \\
&=&\frac{cd(d-c)}{4}\dint\limits_{0}^{1}\frac{(1-2s)}{B_{s}^{2}}\left[ \frac{%
\partial f}{\partial s}\left( a,\frac{cd}{B_{s}}\right) +\frac{\partial f}{%
\partial s}\left( b,\frac{cd}{B_{s}}\right) \right] ds \\
&&+\frac{cd(d-c)}{2}\dint\limits_{0}^{1}\dint\limits_{0}^{1}\frac{(2s-1)}{%
B_{s}^{2}}\frac{\partial f}{\partial s}\left( \frac{ab}{A_{t}},\frac{cd}{%
B_{s}}\right) dtds
\end{eqnarray*}%
Thus, again by integration by parts in the right hand side of (\ref{2-3a}),
it follows that 
\begin{eqnarray*}
&=&\frac{1}{4}(2s-1)\left. \left[ f\left( a,\frac{cd}{B_{s}}\right) +f\left(
b,\frac{cd}{B_{s}}\right) \right] \right\vert _{0}^{1}-\frac{1}{2}%
\dint\limits_{0}^{1}f\left( a,\frac{cd}{B_{s}}\right) +f\left( a,\frac{cd}{%
B_{s}}\right) ds \\
&&+\frac{1}{2}\dint\limits_{0}^{1}(1-2s)\left. f\left( \frac{ab}{A_{t}},%
\frac{cd}{B_{s}}\right) \right\vert
_{0}^{1}dt+\dint\limits_{0}^{1}\dint\limits_{0}^{1}f\left( \frac{ab}{A_{t}},%
\frac{cd}{B_{s}}\right) dtds
\end{eqnarray*}%
\begin{eqnarray*}
&=&\frac{1}{4}\left[ f\left( a,c\right) +f\left( b,c\right) +f\left(
a,d\right) +f\left( b,d\right) \right] -\frac{1}{2}\dint\limits_{0}^{1}f%
\left( a,\frac{cd}{B_{s}}\right) +f\left( a,\frac{cd}{B_{s}}\right) ds \\
&&-\frac{1}{2}\dint\limits_{0}^{1}f\left( \frac{ab}{A_{t}},c\right) +f\left( 
\frac{ab}{A_{t}},d\right) dt+\dint\limits_{0}^{1}\dint\limits_{0}^{1}f\left( 
\frac{ab}{A_{t}},\frac{cd}{B_{s}}\right) dtds
\end{eqnarray*}%
\begin{eqnarray*}
&=&\frac{f(a,c)+f(a,d)+f(b,c)+f(b,d)}{4}+\frac{abcd}{\left( b-a\right)
\left( d-c\right) }\dint\limits_{a}^{b}\dint\limits_{c}^{d}\frac{f(x,y)}{%
\left( xy\right) ^{2}}dxdy \\
&&-\frac{1}{2}\left[ \frac{ab}{b-a}\dint\limits_{a}^{b}\frac{f\left(
x,c\right) }{x^{2}}dx+\frac{ab}{b-a}\dint\limits_{a}^{b}\frac{f\left(
x,d\right) }{x^{2}}dx\right. \\
&&\left. +\frac{cd}{d-c}\dint\limits_{c}^{d}\frac{f\left( a,y\right) }{y^{2}}%
dy+\frac{cd}{d-c}\dint\limits_{c}^{d}\frac{f\left( b,y\right) }{y^{2}}dy%
\right] ,
\end{eqnarray*}%
which completes the proof.
\end{proof}

\begin{theorem}
\label{2.5}Let $f:\Delta =\left[ a,b\right] \times \left[ c,d\right]
\subseteq \left( 0,\infty \right) \times \left( 0,\infty \right) \rightarrow 
\mathbb{R}
$ be a partial differentiable mapping on $\Delta $ with $a<b$ and $c<d$. If $%
\frac{\partial ^{2}f}{\partial t\partial s}\in L(\Delta )$ and $\left\vert 
\frac{\partial ^{2}f}{\partial t\partial s}\right\vert ^{q}$, $q>1,$ is a
harmonically convex function on the co-ordinates on $\Delta $ then one has
the inequality:%
\begin{eqnarray*}
&&\left\vert \frac{f(a,c)+f(a,d)+f(b,c)+f(b,d)}{4}+\frac{abcd}{\left(
b-a\right) \left( d-c\right) }\dint\limits_{a}^{b}\dint\limits_{c}^{d}\frac{%
f(x,y)}{\left( xy\right) ^{2}}dxdy\right. \\
&&-\frac{1}{2}\left[ \frac{ab}{b-a}\dint\limits_{a}^{b}\frac{f\left(
x,c\right) }{x^{2}}dx+\frac{ab}{b-a}\dint\limits_{a}^{b}\frac{f\left(
x,d\right) }{x^{2}}dx\right. \\
&&\left. \left. +\frac{cd}{d-c}\dint\limits_{c}^{d}\frac{f\left( a,y\right) 
}{y^{2}}dy+\frac{cd}{d-c}\dint\limits_{c}^{d}\frac{f\left( b,y\right) }{y^{2}%
}dy\right] \right\vert
\end{eqnarray*}%
\begin{eqnarray*}
&\leq &\frac{ac(b-a)(d-c)}{4bd\left( p+1\right) ^{2/p}} \\
&&\times \left( \frac{C_{1}\left\vert \frac{\partial ^{2}f}{\partial
t\partial s}\left( a,c\right) \right\vert ^{q}+C_{2}\left\vert \frac{%
\partial ^{2}f}{\partial t\partial s}\left( a,d\right) \right\vert
^{q}+C_{3}\left\vert \frac{\partial ^{2}f}{\partial t\partial s}\left(
b,c\right) \right\vert ^{q}+C_{4}\left\vert \frac{\partial ^{2}f}{\partial
t\partial s}\left( b,d\right) \right\vert ^{q}}{4}\right) ^{1/q}
\end{eqnarray*}%
where $A_{t}=tb+(1-t)a$, \ $B_{s}=sd+(1-s)c$ and%
\begin{eqnarray*}
C_{1} &=&\;_{2}F_{1}\left( 2q,1;2;1-\frac{a}{b}\right) \;\times
\;_{2}F_{1}\left( 2q,1;2;1-\frac{c}{d}\right) , \\
C_{2} &=&\;_{2}F_{1}\left( 2q,1;2;1-\frac{a}{b}\right) \;\times
\;_{2}F_{1}\left( 2q,2;3;1-\frac{c}{d}\right) , \\
C_{3} &=&\;_{2}F_{1}\left( 2q,2;3;1-\frac{a}{b}\right) \;\times
\;_{2}F_{1}\left( 2q,1;2;1-\frac{c}{d}\right) , \\
C_{4} &=&\;_{2}F_{1}\left( 2q,2;3;1-\frac{a}{b}\right) \;\times
\;_{2}F_{1}\left( 2q,2;3;1-\frac{c}{d}\right) ,
\end{eqnarray*}%
$\beta $ is Euler Beta function defined by%
\begin{equation*}
\beta (x,y)=\frac{\Gamma (x)\Gamma (y)}{\Gamma (x+y)}%
=\int_{0}^{1}t^{x-1}(1-t)^{y-1}dt,\;\;\;x,y>0,
\end{equation*}%
and is $_{2}F_{1}$ is hypergeometric function defined by%
\begin{eqnarray*}
&&_{2}F_{1}\left( a,b;c;z\right) \\
&=&\frac{1}{\beta (b,c-b)}\int_{0}^{1}t^{b-1}(1-t)^{c-b-1}(1-zt)^{-a}dt,\;%
\;c>b>0,\;\left\vert z\right\vert <1\;\text{(see \cite{abramowitz}) .}
\end{eqnarray*}
\end{theorem}

\begin{proof}
From Lemma (\ref{2.3}), we have%
\begin{eqnarray*}
&&\left\vert \frac{f(a,c)+f(a,d)+f(b,c)+f(b,d)}{4}+\frac{abcd}{\left(
b-a\right) \left( d-c\right) }\dint\limits_{a}^{b}\dint\limits_{c}^{d}\frac{%
f(x,y)}{\left( xy\right) ^{2}}dxdy\right. \\
&&-\frac{1}{2}\left[ \frac{ab}{b-a}\dint\limits_{a}^{b}\frac{f\left(
x,c\right) }{x^{2}}dx+\frac{ab}{b-a}\dint\limits_{a}^{b}\frac{f\left(
x,d\right) }{x^{2}}dx\right. \\
&&\left. \left. +\frac{cd}{d-c}\dint\limits_{c}^{d}\frac{f\left( a,y\right) 
}{y^{2}}dy+\frac{cd}{d-c}\dint\limits_{c}^{d}\frac{f\left( b,y\right) }{y^{2}%
}dy\right] \right\vert \\
&\leq &\frac{abcd(b-a)(d-c)}{4}\dint\limits_{0}^{1}\dint\limits_{0}^{1}\frac{%
\left\vert (1-2t)(1-2s)\right\vert }{\left( A_{t}B_{s}\right) ^{2}}%
\left\vert \frac{\partial ^{2}f}{\partial t\partial s}\left( \frac{ab}{A_{t}}%
,\frac{cd}{B_{s}}\right) \right\vert dtds.
\end{eqnarray*}%
By using the well known Holder inequality for double integrals, $f:\Delta
\rightarrow 
\mathbb{R}
$ is harmonically convex function on the co-ordinates on $\Delta $, then one
has the inequalities:%
\begin{eqnarray*}
&&\left\vert \frac{f(a,c)+f(a,d)+f(b,c)+f(b,d)}{4}+\frac{abcd}{\left(
b-a\right) \left( d-c\right) }\dint\limits_{a}^{b}\dint\limits_{c}^{d}\frac{%
f(x,y)}{\left( xy\right) ^{2}}dxdy\right. \\
&&-\frac{1}{2}\left[ \frac{ab}{b-a}\dint\limits_{a}^{b}\frac{f\left(
x,c\right) }{x^{2}}dx+\frac{ab}{b-a}\dint\limits_{a}^{b}\frac{f\left(
x,d\right) }{x^{2}}dx\right. \\
&&\left. \left. +\frac{cd}{d-c}\dint\limits_{c}^{d}\frac{f\left( a,y\right) 
}{y^{2}}dy+\frac{cd}{d-c}\dint\limits_{c}^{d}\frac{f\left( b,y\right) }{y^{2}%
}dy\right] \right\vert \\
&\leq &\frac{abcd(b-a)(d-c)}{4} \\
&&\times \left( \dint\limits_{0}^{1}\dint\limits_{0}^{1}\left\vert
(1-2t)(1-2s)\right\vert ^{p}dtds\right) ^{1/p}\left(
\dint\limits_{0}^{1}\dint\limits_{0}^{1}\left( A_{t}B_{s}\right)
^{-2q}\left\vert \frac{\partial ^{2}f}{\partial t\partial s}\left( \frac{ab}{%
A_{t}},\frac{cd}{B_{s}}\right) \right\vert ^{q}dtds\right) ^{1/q}
\end{eqnarray*}%
\begin{eqnarray}
&\leq &\frac{abcd(b-a)(d-c)}{4bd\left( p+1\right) ^{2/p}}  \label{2-4a} \\
&&\times \left( \dint\limits_{0}^{1}\dint\limits_{0}^{1}\left(
A_{t}B_{s}\right) ^{-2q}\left( ts\left\vert \frac{\partial ^{2}f}{\partial
t\partial s}\left( a,c\right) \right\vert ^{q}+t(1-s)\left\vert \frac{%
\partial ^{2}f}{\partial t\partial s}\left( a,d\right) \right\vert
^{q}\right. \right.  \notag \\
&&\left. \left. (1-t)s\left\vert \frac{\partial ^{2}f}{\partial t\partial s}%
\left( b,c\right) \right\vert ^{q}+(1-t)(1-s)\left\vert \frac{\partial ^{2}f%
}{\partial t\partial s}\left( b,d\right) \right\vert ^{q}\right) dtds\right)
^{1/q},  \notag
\end{eqnarray}%
where an easy calculation gives 
\begin{eqnarray}
&&\dint\limits_{0}^{1}\dint\limits_{0}^{1}\left( A_{t}B_{s}\right)
^{-2q}tsdtds  \label{2-4b} \\
&=&\frac{1}{4\left( bd\right) ^{2q}}\;\times \;_{2}F_{1}\left( 2q,1;2;1-%
\frac{a}{b}\right) \;\times \;_{2}F_{1}\left( 2q,1;2;1-\frac{c}{d}\right) 
\notag
\end{eqnarray}%
\begin{eqnarray}
&&\dint\limits_{0}^{1}\dint\limits_{0}^{1}\left( A_{t}B_{s}\right)
^{-2q}t(1-s)dtds  \label{2-4c} \\
&=&\frac{1}{4\left( bd\right) ^{2q}}\;\times _{\;2}F_{1}\left( 2q,1;2;1-%
\frac{a}{b}\right) \;\times \;_{2}F_{1}\left( 2q,2;3;1-\frac{c}{d}\right) 
\notag
\end{eqnarray}%
\begin{eqnarray}
&&\dint\limits_{0}^{1}\dint\limits_{0}^{1}\left( A_{t}B_{s}\right)
^{-2q}(1-t)(1-s)dtds  \label{2-4d} \\
&=&\frac{1}{4\left( bd\right) ^{2q}}\;\times \;_{2}F_{1}\left( 2q,2;3;1-%
\frac{a}{b}\right) \;\times \;_{2}F_{1}\left( 2q,1;2;1-\frac{c}{d}\right) 
\notag
\end{eqnarray}%
\begin{eqnarray}
&&\dint\limits_{0}^{1}\dint\limits_{0}^{1}\left( A_{t}B_{s}\right)
^{-2q}(1-t)(1-s)dtds  \label{2-4e} \\
&=&\frac{1}{4\left( bd\right) ^{2q}}\;\times \;_{2}F_{1}\left( 2q,2;3;1-%
\frac{a}{b}\right) \;\times \;_{2}F_{1}\left( 2q,2;3;1-\frac{c}{d}\right) 
\notag
\end{eqnarray}%
Hence, If we use (\ref{2-4b})-(\ref{2-4e}) in (\ref{2-4a}), we obtain the
desired result. This completes the proof.
\end{proof}

\bigskip

\end{document}